\documentclass[10pt]{amsart}
\usepackage{calrsfs}

\usepackage{amsmath,tikz}
\usetikzlibrary{tikzmark,fit}

\usepackage{url}
 \usepackage[greek,english]{babel}
\usepackage{graphicx}
\usepackage{amsthm}
\usepackage{amssymb}
\usepackage{multirow}
\usepackage{tikz-cd}
\usepackage{amsmath}
\usepackage{todonotes}
\usepackage{amsbsy}
\usepackage[all]{xy}
\usepackage{qtree}

\usepackage{enumitem}
\usepackage{xfrac}    
\usepackage{color, colortbl}
\definecolor{LightCyan}{rgb}{0.88,1,1}
\definecolor{Gray}{gray}{0.9}

\usepackage{faktor}

\usepackage{changes}
\definechangesauthor[color=orange,name={Aristidesd Kontogeorgis}]{AK}
\definechangesauthor[color=blue,name={Alex Terezakis}]{AT}

\usepackage{longtable,booktabs}

\usepackage{float} 

\newtheorem{theorem}{Theorem}
\newtheorem{lemma}[theorem]{Lemma}
\newtheorem{corollary}[theorem]{Corollary}
\newtheorem{proposition}[theorem]{Proposition}
\theoremstyle{definition}
\newtheorem{example}[theorem]{Example}
\newtheorem{remark}[theorem]{Remark}
\newtheorem{definition}[theorem]{Definition}

\newcommand{\Z}{\mathbb{Z}}

\newcommand{\Q}{\mathbb{Q}}
\newcommand{\Heis}{\mathrm Heis}
\newcommand{\Fer}{\mathrm Fer}

\newcommand{\Hom}{{\mathrm Hom}}

\newcommand{\Aut}{{ \mathrm Aut }}

\newcommand{\Gal}{\mathrm{Gal}}
\renewcommand{\mod}{{\;\mathrm{mod}}}

\date{\today}

\title{Galois Action on Homology of the Heisenberg Curve.}

\author[A. Kontogeorgis]{Aristides Kontogeorgis}
\address{Department of Mathematics, National and Kapodistrian  University of Athens
Pane\-pist\-imioupolis, 15784 Athens, Greece}
\email{kontogar@math.uoa.gr}

\author[D. Noulas]{Dimitrios Noulas}
\address{Department of Mathematics, National and Kapodistrian University of Athens\\
Panepistimioupolis, 15784 Athens, Greece}
\email{dnoulas@math.uoa.gr}

\date \today

\makeatletter
\newcommand{\aprod}{\mathop{\operator@font \hbox{\Large$\ast$}}}
\makeatother

\begin{document}



\begin{abstract}
The Heisenberg curve is  defined topologically as a cover of the Fermat curve and corresponds to an extension of the projective line minus three points by the non-abelian 
Heisenberg group modulo n. We compute its fundamental group and investigate an action from Artin's Braid group to the curve itself and its homology. We also provide a description of the homology in terms of irreducible representations of the Heisenberg group over a field of characteristic $0$.
\end{abstract}
\maketitle

\section{Introduction}
In \cite{MR4117575}, \cite{MR4186523} the fundamental group of an open abelian Galois cover 
$X \rightarrow \mathbb{P}^1$
of the projective line was computed and used in order to study the actions of the automorphism group $\mathrm{Aut}(X)$, the absolute Galois group $\Gal(\overline{\mathbb{Q}}/\mathbb{Q})$ and the braid group on the homology of the curve $X$. In this article we follow the same approach in order to study the Heisenberg curve, which is a Galois group of the projective line with Galois group the non-abelian discrete Heisenberg group $H_n$, see definition \ref{def:Heiss}. The automorphism of the Heisenberg curve was studied in \cite{MR4252293}. 

The homology group as a $\mathbb{F}[H_n]$-module, over a field $\mathbb{F}$ of characteristic zero, containing the $n$-th roots of unity,  is given in theorem \ref{thm:homHeiss}
\[
H_1(X_H,\mathbb{F}) = \bigoplus\limits_{j=0}^{n-1}\;\;\bigoplus\limits_{i,s=0}^{\gcd(n,j)-1}  \mathbb{F}h_{ijs} \chi_{ijs},
\]
  where the coefficients $h_{ijs} \in \mathbb{Z}$ are explicitly described in eq. (\ref{eq:Hom-coeffs}) and the irreducible characters $\chi_{ijs}$ are discussed in the appendix. 

The approach we use is as follows. From the ramified cover $X\rightarrow \mathbb{P}^1$ we can remove 
the branched points ${0,1,\infty}$ and obtain an open cover 
$X^0 \rightarrow X_3:= \mathbb{P}^1\setminus\{0,1,\infty\}$. Covering space 
theory then provides that the open curve $X^0$ can be described as a quotient of the universal 
covering space $\tilde{X_3}$ by the fundamental group of the open curve $\pi_1(X^0,x^\prime)<\pi_1(X_3,x)$ 
for a fixed point $x \in X_3$ and a randomly chosen fixed preimage $x^\prime$ in $X^0$. The fundamental group 
$\pi_1(X_3,x)$ is the free group $F_2$ of rank $2$ and can be presented as
\[
\pi_1(X_3,x) = \langle x_1, x_2, x_3 | x_1x_2x_3 = 1 \rangle,
\] with each $x_i$ representing the  homotopy 
class of loops on $x$ around the punctures $0,1,\infty$ respectively. 

This setting fits the framework of Y. Ihara as in \cite{Ihara1985-it}, \cite{IharaCruz}. Firstly, the braid group $B_3$ can be realized as a subgroup of $\operatorname{Aut}(F_2)$ generated by the elements $\sigma_i$ for $i=1,2$ given by
\[\sigma_i(x_k) = \begin{cases}
  x_k & k\neq i ,i+1,\\ 
  x_i x_{i+1}x_i^{-1}& k=i, \\
  x_i & k=i+1.
\end{cases}\] where we use that $x_3 = (x_1x_2)^{-1}$. There is a natural surjection $B_3 \rightarrow S_3$ with its kernel being the so called pure 
braid group $P_n$, where every element $\sigma\in P_n$, satisfies
\[\sigma(x_k) \sim x_k^{N(\sigma)}, \; \textrm{ for some }N(\sigma) \in \Z^* = \{\pm1\},\] for $N(\sigma)$ not depending on $x_k$ and by $\sim$ we denote conjugation. According to Ihara the pure braid group is
 a discrete analogue of $\Gal(\overline{\Q}/\Q)$ in the following sense. Fixing a prime $\ell$, by considering the étale pro-$\ell$ 
 fundamental group of $\mathbb{P}^1_{\overline{\Q}}\setminus \{P_1,P_2,\infty\}$, with $P_i \in \Q$, he introduced the monodromy representation
\[\operatorname{Ih}_2 : \Gal(\overline{\Q}/\Q) \longrightarrow \operatorname{Aut}(\mathfrak{F}_2),\] 
with $\mathfrak{F}_2$ being the pro-$\ell$ completion of $F_2$. The group $\mathfrak{F}_2$ can be a considered as a quotient in the pro-$\ell$ category of the 
free pro-$\ell$ group $\mathfrak{F}_3$ and admit a similar presentation $\mathfrak{F}_2 = \langle x_1,x_2,x_3 | x_1x_2x_3=1\rangle$. Ihara's representation has image inside the group
\begin{equation}
\label{eq:IharaAction}
\left\{ \sigma \in \operatorname{Aut}(\mathfrak{F}_2): \; \sigma(x_i)\sim x_i^{N(\sigma)}, (1\leq i \leq 3) \textrm{ for some } N(\sigma) 
\in \Z_{\ell}^*  \right\},
\end{equation} where again $N(\sigma)$ does not depend on $x_i$ and the composition $N\circ \operatorname{Ih}_2:\Gal(\overline{\Q}/\Q)\rightarrow \Z^*_{\ell}$ coincides with the cyclotomic character $\chi_{\ell}$.

We begin by considering the perhaps most famous curve in number theory, that is the Fermat curve in projective 
coordinates $x^n +y^n = z^n$, which has a fundamental group denoted by $R_{\Fer_n}$. By considering a subgroup of $R_{\Fer_n}$ that is the kernel of the surjection $F_2 \rightarrow H_n$ we obtain the Heisenberg curve through the covering subspaces Galois correspondence of $\pi_1(X_3,x)$. This kernel denoted by $R_{\Heis_n}$ is precisely the fundamental group of 
the Heisenberg curve. We then use the free generators of $R_{\Fer}$ as stepping stones and employ the 
Schreier's lemma on them to compute the free generators of $R_{\Heis}$, in order to investigate the braid group action. 

What happens is also interesting in the field of moduli versus field of definition point of view as in the work of Debes, 
Douai in \cite{DebesDouai97} and Debes, Emsalem in \cite{DebesEmsalem}. Let $K$ be a field and $K_S$ a separable closure, then a 
curve defined a priori on $K_S$ might not always be definable over $K$. Their work provide a cohomological 
measure to when this is possible, considering reductions to covers and their automorphisms while descending from $K_S$ to $K$. 

In that framework, similar to how Ihara derives the monodromy representation, an action of $\Gal(K_S/K)$ is lifted on the geometric fundamental group $\Pi_{K_S}(\mathbb{P}^1\setminus\{0,1,\infty\}) \simeq \mathfrak{F}_2$ via the exact sequence
\[1\longrightarrow \Pi_{K_S}(\mathbb{P}^1\setminus\{0,1,\infty\}) \longrightarrow \Pi_{K}(\mathbb{P}^1\setminus\{0,1,\infty\}) \longrightarrow \Gal(K_S/K)\rightarrow 1,\]
which induces an action on the covers of $\mathbb{P}^1$. The field of moduli for a cover $X$ is defined to be the fixed field of the automorphisms in $\Gal(K_S/K)$ that produce a cover isomorphic to $X$.

In contrast to this arithmetic action on $\mathfrak{F}_2$, in this paper we consider the geometric in nature action of the braid group and its induced action on the covers. The Heisenberg curve turns out to be an interesting example in this setting, 
whereas the Fermat curve having a fundamental group that 
is a characteristic subgroup of $F_2$ stays invariant under the braid group action. We provide a case where the Heisenberg curve under the braid action gets mapped to an entirely new non-isomorphic curve.


\textbf{Structure of the paper.} In section \ref{sec:Heisenberg Curves} we provide the preliminaries in \ref{subsec: prelim} and define the Heisenberg curve as a cover of the Fermat curve in \ref{subsec:HeisCoverFer}. Then in \ref{subsec:Fundamental} we compute its fundamental group and describe the Galois action in \ref{subsubsec:GalAction}. Afterwards, in 
\ref{subsec:decomp} we describe the elements that correspond to lifts of homotopy classes of loops around $\infty$ in terms of our previously established generators. Moreover, in \ref{subsec:stabilized} we discuss all the elements of lifts 
around the punctures and define the homology of the compactified curve after removing those points. In \ref{subsec:braid} we investigate the braid action on the curve and its homology and discuss the Burau representation in \ref{subsubsec:burau}. Finally, in \ref{sec:Alexander} we discuss the theory of Alexander Modules leading up to the proof of our main theorem. We provide an appendix \ref{appendix} of the irreducible characters of $H_n$ as many computations in section \ref{sec:Alexander} rely on them and we also provide the following small case example, highlight many parts of the paper.

\begin{example}
  For $n=3$ the Fermat curve $X_{F_3}$ is given by the affine equation $x^3 + y^3 = 1$ and it has a projective canonical weierstrass equation 
  $zy^2 = x^3 -432 z^3$ with genus $1$. The Heisenberg curve $X_{H_3}$ is also of genus $1$, thus we have an isogeny of elliptic curves 
  $X_{H_3} \rightarrow X_{F_3}$ and an equation of $X_{H_3}$ has been computed in \cite{MR4252293} to be 
  $y^2=x^3 + 2^4\cdot 3^6$. 
  
  In terms of group theory, the open Fermat and Heisenberg curves can be defined as quotients of the universal covering space $\tilde{X}_3$ of 
  $X_3 = \mathbb{P}^1\setminus\{0,1,\infty\}$. More precisely, if $F_2 = \langle a,b\rangle$ is the 
  Galois group of $\tilde{X}_3$ over $X_3$, then the open curves $X_{F_3}^\circ,X_{H_3}^\circ$ are defined by the fundamental groups
  \[\langle a^3,b^3, [a,b]\rangle, \ \langle a^3,b^3, [a,[a,b]], [b,[a,b]]\rangle\] respectively, as subgroups of $F_2$. Quotiening these 
  by $\Gamma = \langle a^3, b^3, (ab)^3 \rangle$ which are the third 
  powers of the classes of loops around the branched points $\{0,1,\infty\}$, we obtain the fundamental groups of the curves $X_{F_3},X_{H_3}$ which we expect to be isomorphic to $\Z \times \Z$. Indeed, in \cite{MR4186523} the free generators of the fundamental group of the Fermat curve have been computed and for $n=3$ the generators of $H_1(F_3,\Z)$
   are $[a,b] \mod \Gamma$ and $[a,b]^a \mod \Gamma$. We provide a SageMath \cite{sagemath9.8} script\footnote{The script can be found at 
   \url{https://github.com/noulasd/HeisenbergCurve}} which verifies the word problem that $[a,b],[a,b]^a$ indeed 
   commute $\mod \Gamma$.

  Let $T=[a,b]$ and denote by $x^y = yxy^{-1}$, using the results from this paper we compute that the free generators of the fundamental group of $X_{H_3}^\circ$ are 
  \[ (a^3)^{b^i T^j}, \ (b^3)^{a^i T^j},\ \quad 0\leq i,j\leq 2, \]
  \[ ((ab)^3)^{a^iT^j}, \quad 0\leq i,j\leq 2, \ (i,j)\neq (0,0),\]
  \[ [a,T], \ [a,T]^a. \] 
  Thus, $H_1(X_{H_3},\Z)$ is generated by the classes of $[a,T]$ and $[a,T]^a$ modulo $\Gamma$ and SageMath can verify that also these commute. Lastly, for a 
  field $\mathbb{F}$ of characteristic $0$ which contains the roots of $x^3-1$, if $\chi_{ijs}$ 
  are the irreducible characters of the discrete Heisenberg 
  group $H_3$ as described in the appendix, the regular representation is
  \[\mathbb{F}[H_n] = \bigoplus\limits_{i,s=0}^{2} \mathbb F \chi_{i0s} \oplus  
  \mathbb F 3\chi_{010} \oplus \mathbb{F}3 \chi_{020},\] of 9 one-dimensional and 2 three-dimensional irreducible 
  representations. The homology over $\mathbb{F}$ as a subrepresentation has the character $\chi_{ijs}$ appearing 
  $3/\gcd(3,j) - z_j(i,s)$ times, for $z_j(i,s)$ as defined in the Alexander Modules section, that is we recover 
  \[
  H_1(X_{H_3},\mathbb{F}) = \mathbb{F} \chi_{101} \oplus \mathbb{F} \chi_{202},
  \] as the two-dimensional vector space of the torus that corresponds to the elliptic curve.
  \end{example}

\noindent {\bf Aknowledgements}
The research project is implemented in the framework of H.F.R.I Call “Basic research Financing (Horizontal support of all Sciences)” under the National Recovery and Resilience Plan “Greece 2.0” funded by the European Union Next Generation EU(H.F.R.I.  
Project Number: 14907.
\begin{center}
\includegraphics[scale=0.4]{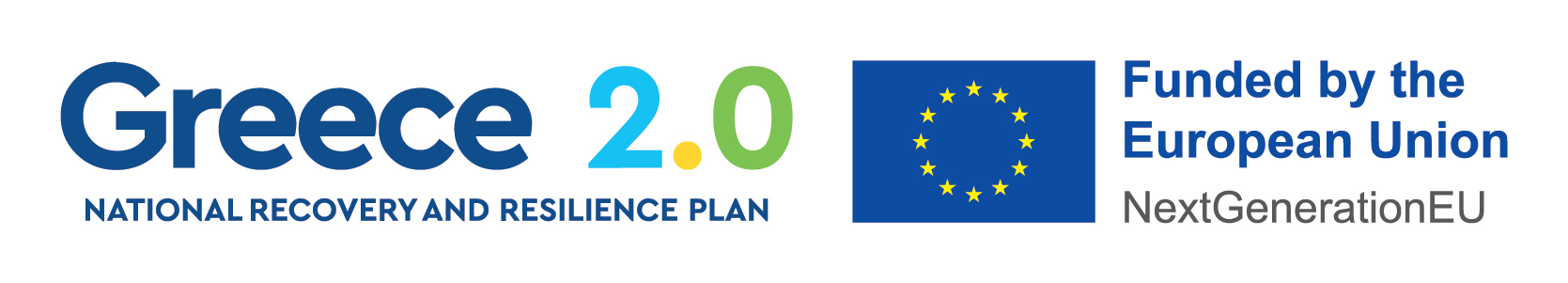}
\hskip 1cm
\includegraphics[scale=0.05]{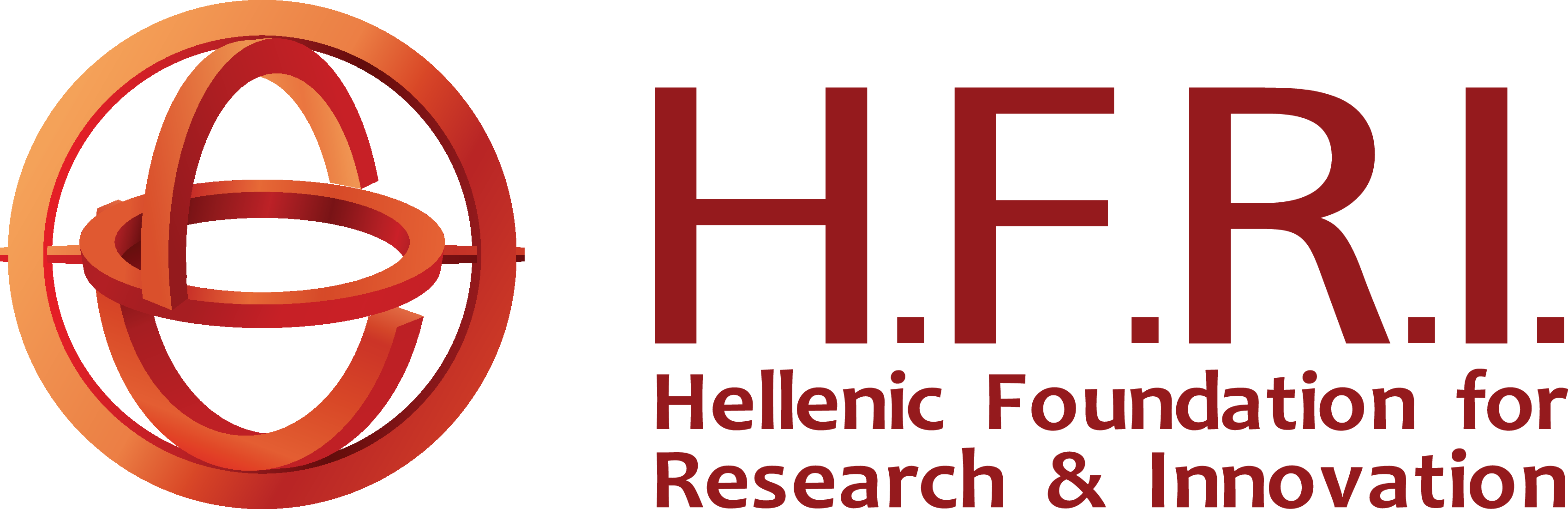}
\end{center}%
\section{Heisenberg Curves} \label{sec:Heisenberg Curves}
\subsection{Group theory preliminaries} \label{subsec: prelim}

\begin{definition} The commutator $[x,y]$ of two elements $x,y$ in a group is defined as $[x,y] = xyx^{-1}y^{-1}$.
\end{definition}

\begin{definition} The exponent $x^y$ of two elements $x,y$ in a group is defined as $x^y = xyx^{-1}$.
\end{definition}

\begin{definition} For a group $G$ the subgroup $G^\prime$ is generated by the commutators $[g,h]$ for $g,h$ in $G$. The abelizanization is defined as $G^{\textrm{ab}} := G/G^\prime$.
\end{definition}
\begin{lemma}\label{gtheory1}
  With the above definitions, for three elements $x,y,z$ in a group, the following identity holds.
  \[[xy,z] = [y,z]^x \cdot [x,z]. \] 
\end{lemma}
\begin{proof}
\begin{align*}
    [xy,z] &= xyz \cdot y^{-1}x^{-1}z^{-1} = xyz \cdot (y^{-1}z^{-1}) \cdot (zy) \cdot y^{-1}x^{-1}z^{-1}\\
    &= x[y,z]zx^{-1}z^{-1} = x[y,z]x^{-1} [x,z]= [y,z]^x \cdot [x,z].
\end{align*}
\end{proof}

\begin{lemma}
  \label{trick} For two elements $x,y$ in a group and for a positive integer $j$
\[[x^j,y] = [x,y]^{x^{j-1}} \cdot [x,y]^{x^{j-2}} \cdots [x,y]^x \cdot [x,y],\]
\[[x,y^j] = [x,y] \cdot [x,y]^y \cdots [x,y]^{y^{j-2}} \cdot [x,y]^{y^{j-1}}.\]
\end{lemma}
\begin{proof}
  See \cite[0.1, p.1]{DDMS}
\end{proof}

\begin{definition}
  \label{def:Heiss}
  The (Discrete) Heisenberg group modulo $n$ is defined as the group of matrices of the form
  \[H_n := \left\{\begin{pmatrix}
    1 & x & z \\
    0 & 1 & y \\
    0 & 0 & 1
  \end{pmatrix}, \quad x,y,z \in \Z/n\Z \right\}.\]
  It is generated by 
  \[ a = \begin{pmatrix}
    1 & 1 & 0 \\
    0 & 1 & 0 \\
    0 & 0 & 1
  \end{pmatrix}, \quad b = \begin{pmatrix}
    1 & 0 & 0 \\
    0 & 1 & 1 \\
    0 & 0 & 1
  \end{pmatrix}.\] Note that, 
  \[[a,b] =  \begin{pmatrix}
    1 & 0 & 1 \\
    0 & 1 & 0 \\
    0 & 0 & 1
  \end{pmatrix}. 
  \] 
\end{definition}

\noindent
According to \cite{MR1816711}, $H_n$ admits a presentation:
\[H_n = \langle a,b \mid a^n, b^n, \ a[a,b]=[a,b]a,  \ b[a,b] = [b,a]b 
\rangle. \] It is evident from the matrices that
we have the extra relation $[a,b]^n = 1$. We can derive it from the other relations as follows, since it will be useful later.

\begin{equation} \label{Tn0}
   1 = a^n b^n = b \cdot a^n [a,b]^n b^{n-1} = b[a,b]^n b^{n-1}= [a,b]^n.
\end{equation}

\subsection{The Heisenberg curve as a cover of the projective line and the Fermat curve} \label{subsec:HeisCoverFer}

\noindent
In \cite{MR4252293} the Heisenberg curve is defined as a cover of the 
projective line minus three points $\mathbb{P}^1 \setminus \{0,1,\infty\}$
 with deck group $H_n$. The Heisenberg curve is a $\Z /n \Z$-cover 
 of the Fermat curve, which in turn has deck group $\Z/n\Z \times 
 \Z/n\Z$ over the punctured projective line. We denote as $F_2$ the 
 free group $\pi_1(\mathbb{P}^1\setminus\{0,1,\infty\}, x_0)$
and by abusing notation we denote the generators by $a$ and $b$, 
which are  the classes of loops around $0$ and $1$ respectively. The whole data is depicted below in the following exact sequences:

\[
\begin{tikzcd}
  0 \arrow[r] & R_{\Heis_n} \arrow[d, hook] \arrow[r] & F_2 \arrow[d, no head] \arrow[d, no head, shift left] \arrow[r] & H_n \arrow[d, two heads] \arrow[r] & 0 \\
  0 \arrow[r] & R_{\Fer_n} \arrow[r]                  & F_2 \arrow[r]                                                   & \Z /n \Z \times \Z /n \Z \arrow[r] & 0
  \end{tikzcd}
\]
 where $R_{\Heis_n}, R_{\Fer_n}$ are the normal closures generated by the elements
\begin{align*}
  R_{\Fer_n} &:= \langle a^n, b^n, \ [a,b] \rangle \\
  R_{\Heis_n} &:= \langle a^n, b^n, \ a[a,b]a^{-1}[a,b]^{-1}, \ 
  b[a,b]b^{-1}[a,b]^{-1} \rangle 
\end{align*} and $[a,b]^n \in R_{\Heis_n}$. We will omit the $n$ and write 
$R_{\Heis}$ whenever it is clear in the given context, although because of the following results from \cite{MR4252293} we will have to keep track of it at various times. 

\begin{lemma}\label{pregenus} The Heisenberg curve is an unramified cover of the Fermat curve if $n$ is odd and a ramified one if $n$ is even. Moreover, in the ramified case the points above $\infty$ have ramification index $2n$.
\end{lemma}

\begin{proof} see \cite[lemma 11]{MR4252293} for the first statement and \cite[proof of lemma 14]{MR4252293} for the second statement.
\end{proof}

Using  the Riemann-Hurwitz genus formula we obtain the following
\begin{lemma} \label{genus}
  The genus $g$ of the (closed) Heisenberg curve is
  \[ g = \begin{cases} \frac{n^2(n-3)}{2} + 1 & \textrm{if } \gcd(n,2)=1, \\
    \frac{n^2(n-3)}{2} + \frac{n^2}4 + 1 & \textrm{if } 2 \mid n.
  \end{cases} \]
\end{lemma}
\begin{proof} See \cite[lemma 15]{MR4252293}.
\end{proof} 


\subsection{Fundamental group of the Heisenberg curve} \label{subsec:Fundamental}

In this section we will describe the fundamental group of the Heisenberg curve. To do this, we will make use of the Schreier lemma \cite[chap. 2, sec. 8]{bogoGrp} 
which given a free group and a generating set of it, it can provide a generating set for a given subgroup under some conditions.

More precisely, for a free group $F$ with basis $X = \{x_1,\ldots,x_s\}$ 
a Schreier Transversal of a subgroup $H$ is a set $T$ of reduced words such 
that all initial segments of a word in $T$ is also in $T$ and for every coset of $H$ in $F$ contains a unique word of $T$. We will denote this unique word by $\overline{g}$ for every $g$ in $F$. The Schreier's lemma concludes that $H$ has a freely generating set consisting of the elements $\gamma(t,x) := tx\overline{tx}^{-1}, t \in T, x \in X$, whenever $tx$ is not in $T$ and $\gamma(t,x)$ does not reduce to $1$.

We will use the description of the fundamental group of the Fermat curve $R_{\Fer}$ as it is given in \cite{MR4186523}, 
and compute $R_{\Heis}$ using the Schreier's lemma. The group $R_{\Fer}$ has the description 
\[ R_{\Fer}=  \langle a_1, b_1, \ldots, a_g,b_g, \gamma_1,\ldots,\gamma_{3n} \ \mid \ \gamma_1\gamma_2 \cdots \gamma_{3n}\cdot [a_1,b_1][a_2,b_2] \cdots [a_g,b_g] = 1\rangle, \] where $g = \frac{(n-1)(n-2)}2$ is the genus 
of the Fermat curve. The $3n$ elements $\gamma_m$ are the elements $(a^n)^{b^i}, (b^n)^{a^i}$ 
and $((ab)^n)^{a^i}$ for $0 \leq i \leq n-1$ and the $2g$ elements remaining are $[b,a]^{a^ib^j}$ for $0\leq i \leq n-2$ and $0\leq j \leq n-3$. Under a base change in \cite{MR4186523}, which even though is written in additive notation only the group operation is being used, we have the free generators of the open Fermat curve:
\[  (a^n)^{b^i}, (b^n)^{a^i}, \; 0\leq i \leq n-1, \; \; [a,b]^{a^ib^j}, \; 0\leq i,j,\leq n-2, \] which will be our initial input in the Schreier process.

As the Heisenberg curve is a $\Z /n \Z$-cover of the Fermat curve, $R_{\Fer_n} / R_{\Heis_n}$ is cyclic, generated by the commutator $T:=[a,b]$. 
Therefore we choose $T^k, 0\leq k \leq n-1$ as a Schreier transversal for $R_{\Fer_n}/R_{\Heis_n}$.
Before computing the generators in Schreier's algorithm, the following elementary lemma will be beneficial in capturing the information about the ramification throughout this paper. 
\begin{lemma} \label{ramification}
   In $H_n$ we have that $(ab)^n = T^{\frac n2} = T^{-\frac n2}$ if $2\mid n$, otherwise it is $1$. Moreover, $(ab)^n$ is in $R_{\Heis_n}$ if and only if $2$ does not divide $n$. 
\end{lemma}
\begin{proof}
In $H_n$ the elements $a$ and $b$ commute with $T$, thus we can compute
\begin{align*}
(a b)^n &= T b \underbrace{a^2 b}_{\textrm{swap}} (a b)^{n-2} = T^{1+2} 
b^2 \underbrace{a^3 b}_{\textrm{swap}} (a b)^{n-3} = \cdots \\
&=  T^{1+\cdots + (n-1)} b^{n-1} a^n b \\
&= T^{\frac{n(n-1)}2} \\
&= \begin{cases}0, & \textrm{ if }\gcd(n,2)=1 \\
  T^{\frac{n}2}, & \textrm{ if } 2\mid n.
\end{cases} 
 \end{align*} For the second assesment, note that $H_n \cong F_2/R_{\Heis_n}$.
\end{proof} Now we compute:
\begin{align*}
    \overline{T^k (b^n)^{a^i}} &= T^k \\
    \overline{T^k (a^n)^{b^i}} &= T^k \\
    \overline{T^k [a,b]^{a^i b^j}} &=  \begin{cases}
        T^{k+1}, & k \leq n-1 \\
        1, & k=n-1
        \end{cases}
\end{align*}
The above computations are straightforward, since they happen in $H_n$.
We compute now the free generators:
\begin{align*}
  T^k (b^n)^{a^i} \cdot \left( \overline{T^k (b^n)^{a^i}}\right)^{-1} & =  
  T^k (b^n)^{a^i} T^{-k}, \quad 0\leq i,k \leq n-1 \\
  T^k (a^n)^{b^i} \cdot \left( \overline{T^k (a^n)^{b^i}}\right)^{-1} & =  
  T^k (a^n)^{b^i} T^{-k}, \quad 0\leq i,k \leq n-1 \\
  T^k [a,b]^{a^i b^j} \cdot \left(\overline{T^k [a,b]^{a^i b^j}}\right)^{-1} &= 
  \begin{cases}
      T^{n-1} [a,b]^{a^i b^j}, &0\leq i,j\leq n-2,\\
      & k= n-1\\
      & \\  
      T^k [a,b]^{a^i b^j} T^{-(k+1)}, & 0\leq k,i,j \leq n-2, \\
      & (i,j)\neq (0,0),\\
      & \\
      1, & (i,j)=(0,0), 0\leq k \leq n-2.
  \end{cases}
  \end{align*}
  
  \begin{lemma}
      The generators of the free group $R_{\Heis_n}$ are listed below, as a union of the following sets:
      \begin{align*}
          A_1 & = \{T^k (a^n)^{b^i} T^{-k}, \quad 0\leq i,k \leq n-1 \}, \quad \# A_1 = n^2, \\
          A_2 & = \{T^k (b^n)^{a^i} T^{-k}, \quad 0\leq i,k \leq n-1 \}, \quad \# A_2 = n^2, \\
          A_3 & = \{T^{n-1} T^{a^i b^j}, \quad 0\leq i,j\leq n-2\}, \quad \# A_3 = (n-1)^2, \\
          A_4 & = \{ T^k T^{a^i b^j} T^{-(k+1)}, \quad 0\leq k,i,j\leq n-2,\; (i,j) \neq (0,0)\}, \\
           \# A_4 & = (n-1)^3 - (n-1).
        \end{align*}
  \end{lemma}
\begin{proof}
    This is a direct consequence of the Schreier lemma. Notice that the above given sets add up to $n^3+1$ generators, as predicted by the Schreier index formula:
    \begin{align*}
    rank(R_{\Heis_n}) &= [F_2 : R_{\Heis_n}](2-1)+1  \\
    &= [R_{\Fer_n}:R_{\Heis_n}](rank(R_{\Fer_n}) -1) + 1 \\
    &= n^3 + 1 
    \end{align*}
    Indeed, we compute $ \sum \#A_i  =  n^3+1$.
\end{proof}
This basis will have a convenient form once Galois action is introduced, 
although it lacks ramification data. Specifically, we would like to be able to describe 
generators as homotopy classes of loops on the punctured tori 
$R_{\Heis}$ consists of, which would mean to have conjugates of $(ab)^n$ 
in the basis. Lemma \ref{ramification} tells us that $(ab)^n$ will 
be in $R_{\Heis_n}$ for odd $n$ and we can expect $(ab)^{2n}$ to be in $R_{\Heis_n}$ for even $n$. We will provide this in detail in a later section.

\subsubsection{Galois action} \label{subsubsec:GalAction}

Suppose we have a group $G$ with a normal subgroup $N$, on which $G$ acts by conjugation. 
We would like to define a conjugation action of $G/N$ to $N$, induced from the action of $G$, 
but this can only be well-defined modulo inner automorphisms. Considering this problem, we can have a well-defined action of $G/N$ on the abelianization $N/N^\prime$. We will use this on the exact sequence 
\[1\rightarrow R_{\Heis} \rightarrow F_2 \rightarrow H_n \rightarrow 1\]
 to have a Galois action on $R_{\Heis}^{\operatorname{ab}}$. 
Consider the two generators $\alpha = a R_{\Heis}$, $\beta = b R_{\Heis}$ 
as well as $\tau = [a,b]R_{\Heis}$ of the group $H_n$. Then, there exists a 
well-defined action
of these three on $R_{\Heis_n}/R_{\Heis_n}^\prime$ given by conjugation, that is

\[ x^{\alpha} = x^a = axa^{-1}, \quad x^{\beta} = x^b = bxb^{-1}, \quad x^\tau = x^T = TxT^{-1}\]
for all $x \in R_{\Heis_n}/R_{\Heis_n}^\prime$. Notice that this is an action, 
which implies that
\[
\left(x^\alpha\right)^\tau = x^{ \tau \alpha } = x^{ \alpha \tau } = \left(x^{\tau}\right)^{\alpha} 
\]
\[
\left(x^\beta\right)^\tau = x^{\tau \beta} = x^{  \beta  \tau} = \left(x^{\tau}\right)^{\beta}, 
\]
that is the actions of $\alpha$ and $\beta$ commute with $\tau$. 
In general, the actions of $\alpha$ and $\beta$ do not commute in this case, as it happens in the Fermat curve. We will still use $a,b,T$ as exponents to denote the conjugation as defined earlier and we use $\alpha,\beta,\tau$ when 
the base of the exponent is in $R_{\Heis}$, when we can realize this as an action from an element of $H_n$. For example, to make sense why this 
is necessary, we have that $T$ is not in $R_{\Heis}$ and we can write $[a^n,T]$ as both $(a^n) - (a^n)^{\tau}$ and as $T^{a^n} 
\cdot T^{-1}$. Could the later one be realized as an action of $H_n$, the element would reduce to the identity.

\begin{lemma} \label{2elements}
  The elements $T^n,T^{-n} [a^ib^j,T]$ are in $R_{\Heis}$. In $R_{\Heis}/R_{\Heis}^\prime$ they are decomposed as follows:
\begin{align}
\label{Tplus}
  T^n &= a^n - (a^n)^\beta - \sum\limits_{k=0}^{n-2}\sum\limits_{i=0}^{n-2-k} [a,T]^{\tau^k \alpha^i}
\\
 \label{Tminusn}
  T^{-n} &= b^n - (b^n)^\alpha - \sum\limits_{k=0}^{n-2}\sum\limits_{j=0}^{n-2-k} [b,T^{-1}]^{\tau^{-k} \beta^j}
\\
 [a^ib^j,T] &= [b,T]^{\alpha^i \sum\limits_{\lambda=0}^{j-1} \beta^\lambda} + [a,T]^{\sum\limits_{\lambda=0}^{i-1} \alpha^\lambda }.
\end{align}
\end{lemma}

\begin{proof}
  The element $T^n$ happens to be in $A_3$ for $(i,j)=(0,0)$, also we expected it to be in $R_{\Heis}$ because it is trivial in $H_n \cong F_2/R_{\Heis}$ from equation \ref{Tn0}. Using lemma \ref{gtheory1} and \ref{trick} we can write
  \[
   [a^ib^j,T]=[b^j,T]^{a^i}[a^i,T] = [b,T]^{\alpha^i (\beta^{j-1} + \beta^{j-2} + \cdots + \beta + 1) } \cdot [a,T]^{ (\alpha^{i-1} + \alpha^{i-2} + \cdots + \alpha + 1)},
  \] which is in $R_{\Heis}$ and the decomposition follows in 
  the abelianization. We also compute:

  \begin{align*} a^n - (a^n)^\beta & = [a^n,b] = T^{a^{n-1}+\cdots + a + 1} \\
    &= [a^{n-1},T]T \cdot [a^{n-2},T]T \cdots [a,T]T \cdot T \\
    &= [a^{n-1},T]T (T^{-1}T) \cdot [a^{n-2},T]T (T^{-2}T^2) \\
    &\cdots [a,T]T (T^{-(n-1)}T^{n-1}) \cdot T (T^{-n}T^n) \\
    &= [a^{n-1},T] + [a^{n-2},T]^\tau + \cdots + [a,T]^{\tau^{n-2}} + T^n
  \end{align*} and the result follows. Similarly, we can prove eq. (\ref{Tminusn}).
\end{proof} 
Notice that in the second and third equality we have the group operation on elements of the form $[*,T]T$ and each $T$ that is not inside the commutator contributes as a conjugation to every 
element that comes after it. As we will use this computation trick again, we will refer to it as {\em nested conjugations}. We rewrite now the sets $A_i$ with $H_n$ action:
\begin{align*} 
A_1 & = \{ (a^n)^{\tau^k \beta^i}, \quad 0\leq i,k \leq n-1 \}\\
          A_2 & = \{(b^n)^{\tau^k \alpha^i}, \quad 0\leq i,k \leq n-1 \} \\
          A_3 & = \{ T^n + [a^ib^j,T]^{\tau^{n-1}}, \quad 0\leq i,j \leq n-2\} \\
          A_4 & = \{ [a^ib^j,T]^{\tau^k}, \quad 0\leq i,j,k \leq n-2, (i,j)\neq (0,0)\}
\end{align*} and we can invertibly write our basis as $0\leq k \leq n-1$:
\[(a^n)^{\tau^k \beta^i}, (b^n)^{\tau^k \alpha^i}, \; 0\leq i \leq n-1, \quad T^n, \quad [a^ib^j,T]^{\tau^k}, \; 0\leq i,j\leq n-2, (i,j)\neq (0,0). \] 
Using lemma \ref{2elements} we decompose in $R_{\Heis}/R_{\Heis}^\prime$ even further:
\begin{align*}
  T^k T^{a^i b^j}T^{-(k+1)} &= [a^i b^j,T]^{\tau^k} = \\
  &=  [b,T]^{\tau^k \alpha^i \sum\limits_{\lambda=0}^{j-1} \beta^\lambda} + [a,T]^{\tau^k \sum\limits_{\lambda=0}^{i-1} \alpha^\lambda },
\end{align*} and 
\begin{align*} 
  T^{n-1}T^{a^i b^j} & = T^n T^{-1} T^{a^i b^j} T^{-1}T= T^n T^{-1} [a^i b^j, T] T \\ 
  &= T^n + [b,T]^{\tau^{n-1} \alpha^i \sum\limits_{\lambda=0}^{j-1} 
  \beta^\lambda} +  [a,T]^{\tau^{n-1}\sum\limits_{\lambda=0}^{i-1} 
  \alpha^\lambda}.
\end{align*}
So far we have proved the following, which will 
be of use when investigating the braid action on the homology
 and the so called Burau representation.

\begin{proposition}\label{prop:basis}
For all $n$, the free $\Z[H_n]$-module $R_{\Heis}^{ab}$ is generated by the elements: \[a^n, b^n, [a,T],[b,T].\] As a $\Z$-module, it can be generated by the $n^3 +1$ elements: $0\leq k \leq n-1$,
\begin{align*}
  (a^n)^{\beta^i, \tau^k}, (b^n)^{\alpha^i \tau^k}, & \quad 0\leq i \leq n-1, \\
  [a,T]^{\alpha^i \tau^k}, & \quad 0\leq i \leq n-3, (i,k)=(n-2,0),\\
  [b,T]^{\alpha^i \beta^j \tau^k}, & \quad 0\leq i \leq n-2, 0 \leq j \leq n-3.
\end{align*}
\end{proposition}
\begin{proof} Follows by counting the indices that are possible to appear from the previous decomposition of $[a^ib^j,T]$. The element $[a,T]^{\alpha^{n-2}}$ follows from the decomposition of $T^n$.
\end{proof}

\subsection{Decomposition of $(ab)^n$} \label{subsec:decomp}

The element $(ab)^n$ corresponds to the $n$-lift of the path 
in $\mathbb{P}^1\setminus\{0,1,\infty\}$ around $\infty$, thus it is interesting to see how it 
decomposes in $R_{\Heis}$ in terms of our generators. From the ramification lemma \ref{ramification} this will only be possible for odd $n$, and for even $n$ we will be able to decompose in $R_{\Heis}$ the elements $(ab)^nT^{-\frac{n}2}, T^{\frac{n}2}(ab)^n$ from which we can form conjugates of $(ab)^{2n}$.
\begin{lemma} For all $n \in \mathbb{N}$, we can write the elements $(ab)^n$ as follows: 
  \[ (ab)^n = \prod\limits_{i=1}^{n-1} \left( \prod\limits_{j=0}^{i-1} [a^ib^{i-1-j},T^{-1}]T^{-1} \right) \; \cdot a^n b^n. \]
 \end{lemma}
 \begin{proof} We begin by computing the $n=2$ case:

  \begin{align*}
    (ab)^2 &= abab^{-1} a^{-1} a bb = a[b,a]a^{-1} aa bb \\
    &= a [b,a] a^{-1} aba^{-1}b^{-1} bab^{-1} \cdot abb \\
    &= a T^{-1} a^{-1} T bab^{-1} \cdot abb \\
    &= [a,T^{-1}]T^{-1} a^2 b^2,
  \end{align*} and the $n=3$ case: 

  \begin{align*}
    (ab)^3 &= (ab)^2 ab = [a,T^{-1}]T^{-1} a^2 b^2 ab \\
    &= [a,T^{-1}]T^{-1} a^2 b \cdot ba \cdot (b^{-1}a^{-1}ab) \cdot b \\
    &= [a,T^{-1}]T^{-1} a^2 b T^{-1} \cdot ab^2 \\
    &=  [a,T^{-1}]T^{-1} a^2 b T^{-1} \cdot (b^{-1}a^{-2}TT^{-1}a^2b) \cdot ab^2 \\
    &= [a,T^{-1}]T^{-1} [a^2b,T^{-1}]T^{-1} \cdot a^2bab^2 
  \end{align*} and 
\[a^2bab^2 = a^2ba \cdot (b^{-1}a^{-1} a^{-2}T T^{-1} a^2 ab) \cdot b^2 = [a^2,T^{-1}]T^{-1}a^3b^3,\] thus 
\[(ab)^3 =   [a,T^{-1}]T^{-1} \cdot  [a^2b,T^{-1}]T^{-1} \cdot [a^2,T^{-1}]T^{-1}a^3b^3, \] and with the nested conjugations trick
\[(ab)^3 = [a,T^{-1}] \cdot [a^2b,T^{-1}]^{T^{-1}} \cdot [a^2,T^{-1}]^{T^{-2}} \cdot T^{-3} \cdot a^3b^3.\] 
  Let
\[ E(k) := \prod\limits_{i=1}^{k-1} \left(\prod\limits_{j=0}^{i-1} [a^ib^{i-1-j},T^{-1}]T^{-1}\right), \] such that
\[ (ab)^k = E(k) a^kb^k. \]
We prove the following claim, for positive integers $\mu, \lambda$:
\[a^\mu b^\lambda ab = \prod\limits_{j=0}^{\lambda-1} [a^\mu b^{\lambda-1-j},T^{-1}]T^{-1} a^{\mu+1} b^{\lambda+1}. \]
Indeed:
\begin{align*} a^\mu b^\lambda ab &= a^\mu b^{\lambda-1} ba (b^{-1}a^{-1}ab) b \\
  &= a^\mu b^{\lambda-1} T^{-1} ab^2 \\
  &= a^\mu b^{\lambda-1}T^{-1} (b^{-(\lambda-1)}a^{-\mu} T T^{-1} a^\mu b^{\lambda-1}) ab^2 \\
  &= [a^\mu b^{\lambda-1}, T^{-1}] T^{-1} a^\mu b^{\lambda-1}ab^2,
\end{align*} and the result follows from recursion. Thus,
\begin{align*}
  (ab)^{k+1} &= (ab)^k ab = E(k) a^kb^k ab \\
  &= E(k) \!\cdot\! \prod\limits_{j=0}^{k-1} [a^k b^{k-1-j},T^{-1}]T^{-1} \cdot a^{k+1}b^{k+1}\\
  &= E(k+1) a^{k+1}b^{k+1}.
\end{align*} 
 \end{proof} We have written $(ab)^n$ into a succesion of elements of the form $[*,T^{-1}]T^{-1}$, so the nested conjugations trick will be applied. Notice that $T^{n(n-1)/2}$ is $\frac{n-1}2 T^n$ 
 in additive notation for odd $n$, but we cannot write this as an element in $R_{\Heis}$ for even $n$. Using the above and that $[*,T^{-1}] = -[*,T]^{\tau^{n-1}}$, we have in $R_{\Heis_n}^{ab}$ for odd $n$:
\[(ab)^n = - \sum\limits_{i=1}^{n-1} \sum\limits_{j=0}^{i-1} [a^ib^{i-1-j},T]^{\tau^{\left(n-1 - \frac{i(i-1)}2 - j\right)}} + \frac{(n-1)}2 T^{-n} + a^n + b^n, \] and for even $n$ we make a slight adjustment to 
the nested conjugations trick
\begin{align*}
(ab)^n T^{-\frac{n}2} &= \prod\limits_{i=1}^{n-1} \left(\prod\limits_{j=0}^{i-1} [a^ib^{i-1-j},T^{-1}]T^{-1} \right) 
\!\cdot\! 
T^{\frac{n(n-1)}2} 
\!\cdot\! 
T^{-\frac{n(n-1)}2} 
\!\cdot\!
 T^{-\frac n2}T^{\frac n2} 
\!\cdot\! a^n b^n 
\!\cdot\! 
 T^{-\frac n2} \\
&= \!-\! \sum\limits_{i=1}^{n-1} \sum\limits_{j=0}^{i-1} [a^ib^{i-1-j},T]^{\tau^{\left(n-1 - \frac{i(i-1)}2 - j\right)}} 
\!+\! 
\left(\frac{n}2-1\right) T^{-n} 
\!+\! (a^n)^{\tau^{\frac n2}}
 \!+\! (b^n)^{\tau^{\frac n2}}
 \!.
\end{align*} 

Using the above formulas and the decomposition of $[a^ib^j,T]$ one could, after some computations, derive a basis with of $R_{\Heis}^{ab}$ with the conjugates of $(ab)^n$ included as generators, but we will not be performing this task. Notice that, in the even case even though $(ab)^n$ is not in $R_{\Heis}$, we expect from the ramification index over the point $\infty$ for $(ab)^{2n}$ to be included, which happens as
\[ (ab)^n T^{-\frac n2} + ((ab)^n T^{-\frac n2})^{\tau^{\frac n2}} = (ab)^{2n},\] and this is stabilized under the action of $\tau^{\frac n2}$, which implies we have the elements $((ab)^{2n})^{\alpha^i \tau^k}, \ 0\leq i \leq n-1, 0\leq k \leq \frac n2 -1$ in $R_{\Heis_n}$ for even n.

\subsection{Elements stabilized} \label{subsec:stabilized}

In this section we will partially describe the fundamental group of the open Heisenberg curve as
\[R_{\Heis_n} = \langle a_1,b_1,\ldots, a_g,b_g, \gamma_1,\ldots,\gamma_m \mid \gamma_1 \gamma_2 \cdots \gamma_m \cdot [a_1,b_1][a_2,b_2]\cdots [a_g,b_g] =1 \rangle,\] where 
$g$ is the genus of the Heisenberg curve and $m$ is the number of branched points of the cover. Combining the Schreier index formula with the above description we have
\[2g + m -1 = n^3 + 1,\] from which we expect that $m$ is $3n^2$ in the unramified case and $\frac 52 n^2$ in the ramified case. 
The decomposition of $(ab)^n$ (or $(ab)^nT^{\frac n2}$ accordingly) suggests that some free generators $[a^ib^j,T]$, or equivalently some conjugates of $[a,T], [b,T]$ can be swapped with the $n^2-1$ conjugates of $(ab)^n$ 
(or accordingly with the conjugates of $(ab)^nT^{\frac n2}$, and in a second base change with the $\frac{n^2}{2}-1$ conjugates of $(ab)^{2n})$. Thus, the above presentation of $R_{\Heis}$ can be realized with $a_i,b_i$ being the elements $[a,T]^{\alpha^i \tau^k}, 
[b,T]^{\alpha^i\beta^j\tau^k}$ for all the indices from proposition \ref{prop:basis} that remain after the supposed base change, and the elements $\gamma_i$ with their stabilizers are depicted in the following table:
\begin{longtable}[c]{c  l  c}
  \toprule
  Invariant element $\gamma_m$ &  Index $m$ & Fixed by\\
  \tabularnewline
  \noindent Unramified case $(n,2)=1$: \\
  \midrule
  \endhead
  $(a^n)^{\tau^k\beta^i}$, 
  $1 \leq k,i \leq n$
  & $1\leq m \leq n^2$  
  & $ \langle \alpha \tau^i \rangle$
  \\
  $(b^n)^{\alpha^i \tau^k}$,
  $1 \leq k,i \leq n$
  & $n^2+1 \leq m \leq 2n^2$ & $\langle \beta \tau^{n-i} \rangle$
  \\
  $((ab)^n)^{\alpha^i \tau^k}$,
  $1 \leq k,i \leq n$
   & $2n^2 + 1 \leq m \leq 3n^2$ & $\langle \alpha \beta \tau^{n-i}\rangle$ \\
  \bottomrule
  \tabularnewline
  Ramified case $(n,2)=2$: \\
  \midrule
  $(a^n)^{\tau^k\beta^i}$,
  $1 \leq k,i \leq n$ 
  & $1\leq m \leq n^2$ & $ \langle \alpha \tau^i \rangle$\\
  $(b^n)^{\alpha^i \tau^k}$,
  $1 \leq k,i \leq n$ & $n^2+1 \leq m \leq 2n^2$  & $\langle \beta \tau^{n-i} \rangle$\\
  $((ab)^{2n})^{\alpha^i \tau^k}$,
  $\substack{1 \leq i \leq n
  \\1 \leq k \leq \frac{n}{2}}$
  & $2n^2 + 1 \leq m \leq \frac 52n^2$ & $\langle \alpha \beta \tau^{n-i}\rangle, \langle \tau^{\frac n2}\rangle$ \\
  \bottomrule
  \end{longtable}

\begin{definition} Let $\Gamma$ be the free $\Z$-module generated by $\langle \gamma_1,\ldots \gamma_{3n^2} \rangle$ in the unramified case or $\langle \gamma_1,\ldots,\gamma_{\frac{5}2 n^2}\rangle$ in the ramified case. 
  Equivalently, we can set $\Gamma$ to be $R_{\Heis}\cap \langle a^n,b^n, (ab)^n\rangle$ in both cases.
\end{definition}

We denote by $X_H$ the Heisenberg curve as the ramified cover $X_H \rightarrow \mathbb{P}^1$ and we have that the homology group of the curve is 

\[H_1 (X_H,\Z) = \frac{ R_{\Heis}/R^{\prime}_{\Heis}}{\Gamma}.\] Having the $\Z[H_n]$-basis of $R^{ab}_{\Heis}$ in mind, the natural question now is if both $[a,T]$ and $[b,T]$ 
generate the homology as a $\Z[H_n]$-module. Although the decomposition of $(ab)^n$ is very complicated and contains many non-free elements, we can extract from it enough information to asnwer this.

\begin{proposition}\label{prop:basishomology}
For $n=3$ the homology of the Heisenberg curve is generated by $[a,T]$ as a $\Z[H_3]$-module. For $n\geq 4$ it is generated by both $[a,T], [b,T]$ as a $\Z[H_n]$-module, this means they belong in different classes $\mod\Gamma$.
\end{proposition}
\begin{proof}
Observe in the decomposition of $(ab)^n$ that for $n=3$ the only non-zero power of $b$ is $1$, which means only one conjugate of $[b,T]$ appears, say $[b,T]^x$. We can apply the action of $x^{-1}$ such that $[b,T]$ appears 
and by reducing $\mod \Gamma$ every other element is of the form $[a,T]^y$, that is $[b,T]$ is linearly dependant on conjugates of $[a,T]$ modulo $\Gamma$. The same reasoning when $n\geq 4$ provides that $[b,T]$ will be linearly 
dependant to conjugates of $[a,T]$ and to at least one of its own conjugates modulo $\Gamma$.
\end{proof}We provide the exact computation when $n=3$, combining the decompositions of $(ab)^n, [a^ib^j,T]$ and $T^n$ we are left with
\[(ab)^n = -[a,T]^{\tau^2} - [b,T]^{\alpha^2 \tau}- [a,T]^{\alpha \tau} - (a^3)^\beta + \beta^3 \] and applying the action of $\alpha \tau^2$ and reducing $\mod \Gamma$ we get
\[ [b,T] \equiv - [a,T]^{\alpha \tau} - [a,T]^{\alpha^2}\mod \Gamma.\] 

\subsection{Braid group action} \label{subsec:braid} 

In this section we will describe the action of the braid group on $H_1(X_H,\Z)$. The braid group $B_3$ in this setting will be realized 
as an automorphism group of $F_2$ in terms of the faithfull Artin representation. See \cite[2.1]{MR4117575} for a more detailed description and 
an application to cyclic covers of the projective line, also \cite{MR4186523} for the action of $B_3$ on the closed Fermat surface. The group $B_3$ is generated by 
$\sigma_1$ and $\sigma_2$ such that
\[\sigma_1 (a) = aba^{-1} \quad \sigma_1 (b) = a \quad \sigma_2 (a) = a \quad \sigma_2 (b) = a^{-1}b^{-1}.\] We would like to have a well-defined action of $B_3$ on the fundamental group of the Heisenberg curve, as it happens with the Fermat curve, 
however there is an obstruction.
\begin{lemma} 
  \label{characteristic} 
  The group $R_{\Heis_n}$ is characteristic for odd $n$, which means every automorphism in $\operatorname{Aut}(F_2)$ keeps $R_{\Heis_n}$ invariant.
\end{lemma}
\begin{proof}
The group $R_{\Heis} \subset F_2 = \langle a,b\rangle$ is the normal closure of the elements $a^n,b^n, [a,T],[b,T]$ and the automorphism group of $F_2$ is generated by the Nielsen transformations 
below, see \cite[Th. 1.5, p.125]{bogoGrp}
\[\Aut (F_2) = \langle n_a, n_b, n_{ab}, n_{ba}\rangle \]
where 
\begin{align*}
 n_a(a)=a^{-1} &&  n_a(b)=b && n_b(a)=a && n_b(b)=b^{-1}  \\
 n_{ab}(a)=ab &&  n_{ab} (b)=b && n_{ba}(a)=a && n_{ba}(b)=ba
\end{align*}
We compute
\begin{align*}
n_a(a^n) &= a^{-n}, \\
n_a([a,T]) &= [a^{-1}, [a^{-1},b]] = [a,T]^{a^{-2}T^{-1}}, \\
n_a([b,T]) &= [b,[a^{-1},b]] = [b,[b,a]^{a^{-1}}] = [T,b]^{a^{-1}}, \\
n_{ab}(a^n) &= (ab)^n \in R_{\Heis_n} \textrm{ in the unramified case,} \\ 
n_{ab}([a,T]) &= [ab, [ab,b]] = [ab, [a,b]] = [b,T]^a \cdot [a,T], \\
n_{ab}([b,T]) &= [b,[ab,b]] = [b,T].
\end{align*} and the action of $n_b, n_{ba}$ is symmetrical.\end{proof}
Notice that because of the anomaly in $n_{ab}(a^n)$, $R_{\Heis_n}$ cannot be characteristic for even 
$n$. In contrast, in the unramified case the previous lemma tells us that the braid group $B_3$ keeps the Heisenberg curve invariant, this means it has a well-defined action on $R_{\Heis}$ which induces an action on $H_1(X_H,\Z)$.

\begin{remark}
There are two things worth pointing out about the ramified case here, one is that $B_3$ keeps $\Gamma$ invariant and thus there is a well-defined action of $B_3$ on $H_1(X_H,\Z)$ nonetheless. The second thing is that since $\sigma_2(b^n) = ((ab)^{-n})^b \notin R_{\Heis}$ the braid group $B_3$ sends the Heisenberg curve to its conjugate curves defined by the orbit of $\sigma_2 (R_{\Heis})$. 
This is an interesting example in the field of moduli versus field of definition perspective, as discussed in \cite{DebesDouai97}, \cite{DebesEmsalem},
where for a base field $K$ a curve $X$ defined a priori on a separable closure $K_s$ might or might not be definable over $K$. For this question the isomorphic curves of $X$ with different models than $X$, after an action of an automorphism, are taken into account.
\end{remark}

To expand further on the previous remark, suppose that we work over $\overline{\Q}$ for simplicity. Following the definitions in \cite{DebesDouai97}, we have that $\Pi_{\overline{\Q}}(B^*)$ is the geometric fundamental group of $B^* = \mathbb{P}^1_{\overline{\Q}}-\{0,1,\infty\}$. The group $\Pi_{\overline{\Q}}(B^*)$ is canonically isomorphic to the profinite free group $\mathfrak{F}_{2}$ and fits the following split exact sequence \[1\rightarrow \Pi_{\overline{\Q}}(B^*) \rightarrow \Pi_{\Q}(B^*) \rightarrow \Gal(\overline{\Q}/\Q)\rightarrow 1, \] with the middle term $\Pi_{\Q}(B^*)$ being the arithmetic fundamental group of $B^*$. Through Galois theory, $G$-covers of $B^*$ correspond to surjective homomorphims $\Psi: \Pi_{\Q}(B^*) \rightarrow G$ and also the well-defined action up to inner automorphisms and sections of $\Gal(\overline{\Q}/\Q)$ on $\Pi_{\overline{\Q}}(B^*)$ induces an action on these covers. 

As the braid group $B_3$ can act on the topological generators of $\mathfrak{F}_2$ it has a similar action on the covers. Let $\sigma$ denote $\sigma_2 \in B_3$ and consider the homomorphisms
\[\Psi, \Psi^\sigma : \mathfrak{F}_2 \rightarrow  H_n, \] corresponding to quotiening by the profinite groups $\hat{R}_{\Heis_n}, \sigma(\hat{R}_{\Heis_n})$ respectively. We provide an interesting fact above these two covers in the next proposition.

\begin{proposition}
  Assume $n$ to be even. The Heisenberg curve and its $\sigma$-conjugate are not isomorphic over $\overline{\Q}$. 
\end{proposition}
\begin{proof} According to \cite{DebesDouai97}, the curves are isomorphic if and only if $\Psi, \Psi^\sigma$ have conjugate images, that is there is an element $h$ in $H_n$ such that
  \[\Psi(x) = h \Psi^\sigma(x) h^{-1}, \ \textrm{ for all }x \in \Pi_{\overline{\Q}}(B^*).\]
  Set $x$ to be $ab$, then $\Psi(ab) = \alpha \beta$ which is of order $2n$ in $H_n$, however, $\Psi^\sigma(ab)$ is of order $n$ since $(ab)^n$ is contained in $\sigma_2(R_{\Heis_n})$. Thus the above criterion is not satisfied and the result follows. 
\end{proof}

We describe now the action of $B_3$ on the $\Z[H_n]$-generators of $H_1(X_H,\Z)$ according to proposition \ref{prop:basishomology}, which will be useful to describe the Burau representation. Firstly, we have that
\[ \sigma_1 (T) = (T^{-1})^a, \quad  \sigma_2 (T)  = [b^{-1},a^{-1}] = (T^{-1})^{b^{-1}a^{-1}},\] and with a repeated use of 
lemma \ref{trick} we compute the following.
\begin{align*}
  \sigma_1 [a,T] &= -[b,T]^{\alpha \tau^{-1}} \\ 
  \sigma_1 [b,T] &= -[a,T]^{\alpha \tau^{-1}} \\ 
  \sigma_2 [a,T] &= [b,T]^{\beta^{-1}\tau^{-1} - \alpha^{-1}\beta^{-1}\tau^{-1}}  - [a,T]^{\alpha^{-1}\tau^{-1}} \\
  \sigma_2 [b,T] &= [b,T]^{(\alpha^{-1}\beta^{-1})^2\tau^{-1}} + [a,T]^{\alpha^{-1}\beta^{-1}\alpha^{-1}\tau^{-1}}.
\end{align*}

\subsubsection{Burau Representation} \label{subsubsec:burau}
We will see that  the action of $B_3$ on the generators of the homology of the Heisenberg curve induces a representation

\begin{equation}
\label{eq:Braid}  
\rho : B_3 \longrightarrow \operatorname{GL}(2,A),
\end{equation}
 where 
\[
  A= \Z\langle a^{\pm1},b^{\pm1}\rangle/ \langle [a,t],[b,t]), t = [a,b]\rangle  
\] 
and  
  $\Z\langle a^{\pm1},b^{\pm1}\rangle$
is the noncommutative Laurent 
polynomial ring in two variables over $\Z$. For a similar construction of the Burau 
representation for cyclic covers of $\mathbb{P}^1$ see \cite{MR4117575}.

Let $X_3 = \mathbb{P}^1\setminus\{0,1,\infty\}$ and denote by $\tilde{X}_3$ its 
universal covering space. We can obtain the Heisenberg curves $X_{H_n}$ as quotients 
of an abstract curve $Y$ defined as $\tilde{X}_3/I$ where $I$ corresponds to $\langle [a,t], [b,t] \rangle$. Then $X_{H_n}$ is the quotient $Y/J$, where $J$ corresponds to $\langle a^n,b^n\rangle$ as depicted in the following diagram.

\begin{center}
\begin{tikzcd}
  \tilde{X}_3 \arrow[rrd, no head] \arrow[rdd, "R_{\Heis_n}", no head] \arrow[ddd, "F_2"', no head] &                                    &                                                                                       \\
                                                                                                    &                                    & Y \arrow[ld, no head] \arrow[lldd, "(\Z \times \Z)\rtimes \Z", no head, bend left=49] \\
                                                                                                    & X_{H_n} \arrow[ld, "H_n", no head] &                                                                                       \\
  X_3                                                                                               &                                    &                                                                                      
  \end{tikzcd}
\end{center}
The homology group $H_1(Y,\Z)$ through similar techniques 
from this paper has a conjugation action by the discrete Heisenberg group  
\[
H=
\left\{
  \begin{pmatrix}
    1 & x & y \\
    0 & 1 & z \\
    0 & 0 & 1 
  \end{pmatrix}: x,y,z\in \mathbb{Z}
\right\} \cong  (\Z \times \Z)\rtimes \Z = \langle 
  a,t\rangle 
  \rtimes
  \langle b \rangle  
\]
 and is thus an $A$-module. Repeating the braid action computations from the previous section we get that this action 
introduces the  representation $\rho$ given in eq. (\ref{eq:Braid}) which is given by
\begin{align*}
  \rho(\sigma_1) &=
   \begin{pmatrix} 
  0 & -at^{-1} \\
  -at^{-1} & 0
\end{pmatrix}, \\ 
\rho(\sigma_2) &= \begin{pmatrix} 
  -a^{-1}t^{-1} & -a^{-1}b^{-1}a^{-1}t^{-1} \\
  (1-a^{-1})b^{-1}t^{-1} & (a^{-1}b^{-1})^2t^{-1}
\end{pmatrix}.
\end{align*}
Similar to \cite{MR4117575} we define 
\[
  \mathbb{Z}_\ell[\overline{ H}] = \lim_{\leftarrow \atop n} \mathbb{Z}_\ell
  [ H_{\ell^n}].
\]
Then, we have that 
\[
  H_1(Y,\Z_\ell)= \hat{R}_{\mathrm{Heis}}/\hat{R}_{\mathrm{Heis}}^{\prime},
\]
where $R_{\mathrm{Heis}}$ is the fundamental group of $Y$ and $\hat{R}_{\mathrm{Heis}_n}$ is its pro-$\ell$ completion, which is given by 
\begin{proposition}
  \label{prop:IharY}
  The group  $H_1(Y, \mathbb{Z}_\ell)$ is generated by both $[a,t], [b,t]$ as a $\Z_\ell[\bar{H}]$-module.
  \end{proposition}
Therefore, each element $\sigma \in \mathrm{Gal}(\overline{\mathbb{Q}}/ \mathbb{Q})$ acts on  $H_1(Y,\mathbb{Z}_\ell)$, using Ihara's action given in eq. (\ref{eq:IharaAction}), on the generators of $H_1(Y, \mathbb{Z}_\ell)$ given in proposition \ref{prop:IharY}. 


\section{Alexander modules} \label{sec:Alexander}
Let $\mathbb{F}$ be a field of characteristic $0$ containing the $n$ different $n$-th roots of unity. In this section we will use the theory of Alexander modules and the Crowell exact sequence, as described in Chapter $9$ from 
\cite{Morishita2011-yw}, to describe the homology $H_1(X_H, \Z)$ of the closed Heisenberg curve as an $\mathbb{F}[H_n]$-module in terms of the characters of the Heisenberg group. 

Let $F_2$ be the free group $\pi_1 (\mathbb{P}^1 \setminus \{0,1,\infty \},x_0)$ as before with generators 
$a,b$ and $R_{\Heis}$ as defined previously. Let $\Gamma$ be $\langle a^n, b^n, (ab)^n \rangle$,  we can describe $G := F_2 / R_{\Heis}\cap\Gamma$ as  

\[ G = \langle a,b, ab \mid \ a^{e_1} =  b^{e_2} = (ab)^{e_3} = a \cdot b \cdot (ab)^{-1} = 1 \rangle.  \] where $e_1 = e_2 = n$ and $e_3 = n$ or $2n$ if $n$ is odd or even respectively.
Let $\overline{R}_{\Heis}:= R_{\Heis}/R_{\Heis}\cap \Gamma \cong R_{\Heis}\cdot \Gamma / \Gamma$. This reduces to $R_{\Heis}/\Gamma$ in the unramified case. We have a quotient map 
\[
\psi : F_2 / \Gamma \rightarrow F_2 / R_{\Heis} \cong H_n.
\] 
Set also $\varepsilon: \Z[H_n] \rightarrow \Z$ to be the augmentation map $\sum a_g g \mapsto \sum a_g$. 

We consider $\mathcal{A}_{\psi}$ to be the {\em Alexander module}, a free $\Z$-module

\[\mathcal{A}_\psi = \left( \bigoplus\limits_{g \in F_2/\Gamma} \Z[H_n] dg \right) / \langle d(g_1g_2) - dg_1 -\psi(g_1)dg_2: \ g_1,g_2 \in F_{2}/\Gamma \rangle_{\Z[H_n]} \] where $\langle \cdots \rangle_{\Z{H_n}}$ is considered to be a $\Z[H_n]$ module generated by the elements appearing inside. 

By the above definitions, $\overline{R}_{\Heis}^{ab}$ is $H_1(X_H,\Z)$. Define the map $\theta_1  :\overline{R}_{\Heis}^{ab} \rightarrow \mathcal{A}_\psi$ given by
\[\overline{R}_{\Heis}^{ab} \ni n \mapsto dn \] and the map $\theta_2 : \mathcal{A}_\psi \rightarrow \Z[H_n]$ to be the homomorphism induced by 

\[dg\mapsto \psi(g)-1 \ \textrm{ for } g \in G.\] We have the following exact sequence
\[
  \begin{tikzcd}
    1 \arrow[r] & \overline{R}_{\Heis} \arrow[r] & G \arrow[r, "\psi"] & H_n \arrow[r] & 1
    \end{tikzcd} 
\]
from which we obtain the Crowell exact sequence of $\Z[H_n]$-modules \cite[sec. 9.2]{Morishita2011-yw}

\begin{equation}\label{Crowell}
\centering
\begin{tikzcd}
  1 \arrow[r] & {\overline{R}_{\Heis}^{ab} = H_1(X_H,\Z)} \arrow[r, "\theta_1"] & \mathcal{A}_\psi \arrow[r, "\theta_2"] & {\Z[H_n]} \arrow[r, 
  "\varepsilon"] & \Z \arrow[r] & 1.
  \end{tikzcd}
\end{equation}

For the following proposition, let $F_3$ be the free group generated by $x_1,x_2,x_3$, we use that there is a natural epimorphism $\pi : F_3 \rightarrow G$ 
mapping $x_1 \mapsto a, x_2 \mapsto b$ and $x_3 \mapsto (ab)^{-1}$.

\begin{proposition} The module $\mathcal{A}_\psi$ admits a free resolution as a $\Z[H_n]$-module: 

 \begin{equation}\label{resolution}
  \begin{tikzcd}
    {\Z[H_n]^4 } \arrow[r, "Q"] & { \Z[H_n]^3} \arrow[r] & \mathcal{A}_\psi \arrow[r] & 0
    \end{tikzcd}
 \end{equation} where $4$ and $3$ appear as the number of relations and generators of $G$ respectively. The map $Q$ is expressed in form of Fox derivatives 
 \cite[sec. 3.1]{BirmanBraids}, \cite[chap. 8]{Morishita2011-yw} as follows, let $\beta_1,\beta_2,\beta_3, \beta_4 \in \Z[H_n]$, then
\[ \begin{pmatrix}\
 \beta_1 \\
 \beta_2 \\ 
 \beta_3 \\
 \beta_4
 \end{pmatrix} \mapsto 
 \begin{pmatrix} \psi \pi (\frac{\partial x_1^{e_1}}{\partial x_1}) & \psi \pi (\frac{\partial x_2^{e_2}}{\partial x_1}) & \psi \pi (\frac{\partial x_3^{e_3}}{\partial x_1}) & \psi \pi (\frac{\partial x_1 \cdot x_2 \cdot x_3}{\partial x_1}) \\
  \psi \pi (\frac{\partial x_1^{e_1}}{\partial x_2}) & \psi \pi (\frac{\partial x_2^{e_2}}{\partial b}) & \psi \pi (\frac{\partial x_3^{e_3}}{\partial x_2}) & \psi \pi (\frac{\partial x_1 \cdot x_2 \cdot x_3}{\partial x_2}) \\
  \psi \pi (\frac{\partial x_1^{e_1}}{\partial x_3}) & \psi \pi (\frac{\partial x_2^{e_2}}{\partial x_3}) & \psi \pi (\frac{\partial x_3^{e_3}}{\partial x_3}) & \psi \pi 
  (\frac{\partial x_1 \cdot x_2 \cdot x_3}{\partial x_3})
 \end{pmatrix} \cdot  \begin{pmatrix}\
  \beta_1 \\
  \beta_2 \\ 
  \beta_3 \\
  \beta_4
  \end{pmatrix}\]
  where $\pi$ is the natural epimorphism $F_3 \rightarrow G$ as defined previously.
\end{proposition}

\begin{proof} See \cite[Cor. 9.6]{Morishita2011-yw}.
\end{proof}

We apply on the exact sequence \ref{Crowell} and the functor $\otimes_{\Z} \mathbb{F}$ to get the exact sequence of $\mathbb{F}[H_n]$-modules
\begin{equation}\label{Crowell2}
  \begin{tikzcd}
    1 \arrow[r] & {H_1(X_H,\Z) \otimes_{\Z} \mathbb{F}} \arrow[r, "\theta_1 \otimes 1"] & \mathcal{A}_\psi \otimes_{\Z} \mathbb{F}\arrow[r, "\theta_2 \otimes 1"] & {\mathbb{F}[H_n]} \arrow[r, "\varepsilon"] & \mathbb{F} \arrow[r] & 1
    \end{tikzcd}
  \end{equation}

  In general, it is well-known that the tensor functor is not left exact, but since $\operatorname{char} \mathbb{F}=0 $ we have that $\mathbb F$ is a flat $\Z$-module and this provides the left exactness. By counting dimensions, we have that 
  \[\dim_{\mathbb{F}} \mathcal{A}_\psi \otimes_{\Z} \mathbb{F} = 2g+n^3 -1 \] where $g$ is the genus as in lemma \ref{genus}. 
  We will describe now the regular representation $\mathbb{F}[H_n]$ in terms of the irreducible characters $\chi_{ijs}$ of $H_n$, see appendix A for their definition in a short survey of them. Recall that every irreducible representation $\chi$ appears $\deg \chi$ times in the decomposition of $\mathbb{F}[H_n]$.

  \begin{equation*} \label{regularrep}
     \mathbb{F}[H_n] = \bigoplus\limits_{j=0}^{n-1} \bigoplus\limits_{i, s=0}^{\gcd(n,j)-1} \mathbb{F} \frac{n}{\gcd(n,j)} \chi_{ijs}.
  \end{equation*}

The method now is to use the free resolution given in eq. (\ref{resolution}) in order to describe $\mathcal{A}_\psi \otimes_\Z \mathbb{F}$ as an $\mathbb{F}[H_n]$-module and then 
use the Crowell exact sequence to understand the homology. The $\Z[H_n]$-module $\mathcal{A}_\psi$ is the cokernel of the map $Q$. 
We will denote as previously $\alpha,\beta,\tau$ the images in $H_n$. We compute


\[ \frac{\partial x_i^{e_i}}{\partial x_j} = \delta_{ij} \left(1 + x_i + x_i^2 + \cdots + x_i^{e_i-1} \right)\]
\[ \frac{ \partial x_1 \cdot x_2 \cdot x_3}{\partial x_1} = 1\]
\[\frac{ \partial x_1 \cdot x_2 \cdot x_3}{\partial x_2} = x_1\]
\[\frac{ \partial x_1 \cdot x_2 \cdot x_3}{\partial x_3} = x_1 x_2\]    

Set the following:

\begin{align*} 
  \Sigma_1 &= 1 + \alpha + \cdots + \alpha^{n-1} \\
  \Sigma_2 &= 1+\beta+ \cdots + \beta^{n-1} \\
  \Sigma_3 &= 1 + (\alpha \beta) + \cdots + (\alpha \beta)^{n-1}
\end{align*}
In the ramified case we are interested in having the following sum instead of $\Sigma_3$

\begin{align*}
\Sigma_3^\prime &= 1 + (\alpha \beta) + \cdots + (\alpha \beta)^{2n-1} = \\ 
&= 1 + (\alpha \beta) + \cdots (\alpha \beta)^{n-1} + \tau^{\frac n2} + \tau^{\frac n2} (\alpha \beta) + \cdots \tau^{\frac n2} (\alpha \beta)^{n-1} = \\
&= (1+\tau^{\frac n2}) \Sigma_3.
\end{align*} Set $\Sigma_3^*$ to be $\Sigma_3$ or $\Sigma_3^\prime$ varying based on the ramification as previously. The map $Q$ in proposition \ref{resolution} is given by the matrix on the left-hand side of the following equation

\begin{equation} \label{Qmap}
  \begin{pmatrix}
    \Sigma_1 & 0 & 0 & 1 \\
    0 & \Sigma_2 & 0 &  \alpha \\
    0 & 0 & \Sigma_3^* & \alpha \beta
  \end{pmatrix}\begin{pmatrix}
    r_1 \\
    r_2 \\ 
    r_3 \\
    r_4
  \end{pmatrix} = \begin{pmatrix} 
    \Sigma_1 r_1 + r_4 \\
    \Sigma_2 r_2 + \alpha r_4 \\
    \Sigma_3^* r_3 + \alpha\beta r_4
  \end{pmatrix}
\end{equation} where $r_i \in \Z[H_n]$. Observe that
\begin{align*}
  \Sigma_1 \alpha &= \Sigma_1, \\
  \Sigma_2 \beta &= \Sigma_2, \\
  \Sigma^*_3 \alpha\beta &= \Sigma^*_3,
\end{align*}in particular, in the even $n$ case

\begin{equation}\label{ramification2}
  \Sigma^\prime_3 \tau^{\frac n2 } = \Sigma^\prime_3.
\end{equation}

\begin{lemma}\label{sigmas} For $i=1,2$ the following isomorphisms hold
  \[
    \Sigma_i \Z[H_n] \cong \Sigma_i \Z[ \Z/n\Z \times \Z/n\Z]
  \] and 
  \[\Sigma_3^* \Z[H_n] \cong \begin{cases} 
    \Sigma_3 \Z[\Z/n\Z \times \Z/n\Z], & \textrm{ in the unramified case}, \\
    \Sigma_3^\prime \Z[\Z/n\Z \times \Z/\frac{n}{2}\Z], & \textrm{ in the ramified case}. \\
  \end{cases}\] 
\end{lemma}

\begin{proof} We cannot work as in the proof of lemma 4.4 in \cite{MR4186523}, because $H_n$ is not a direct product of groups. We will, however, determine the action explicitly. Let $\beta^j \alpha^i \tau^k$ be an element in $H_n$. Then

  \begin{align*}
    \Sigma_1 \beta^j \alpha^i \tau^k &= \Sigma_1 \alpha^i \beta^j \tau^{k-ij} = \Sigma_1 \beta^j \tau^{k-ij}\in \Sigma_1 \Z[\langle \beta \rangle \times \langle \tau \rangle], \\
    \Sigma_2 \beta^j \alpha^i \tau^k &= \Sigma_2 \alpha^i \tau^{k}\in \Sigma_1 \Z[\langle \alpha \rangle \times \langle \tau \rangle].
  \end{align*} 
  For the action of $\Sigma_3^*$ we have to form pairs of $\alpha\beta$ instead of forcing $\alpha^i$ to commute with $\beta^j$, notice that,
  \[ \Sigma_3^* \alpha = \Sigma_3^* \alpha \beta^n = \Sigma_3^* \beta^{n-1}, \] and \[ \Sigma_3^* \beta = \Sigma_3^* \beta \alpha^n = \Sigma_3^* \alpha^{n-1} \tau^{-1}. \] These generalize to
  \[ \Sigma_3^* \alpha^i = \Sigma_3^* \beta^{n-i} \tau^{1+2+\cdots + (i-1)},\]
  \[ \Sigma_3^* \beta^j = \Sigma_3^* \alpha^{n-j} \tau^{-1-2-\cdots - j }.\] 
  Thus, $\Sigma_3^* \Z[\langle \alpha \rangle \times \langle \tau \rangle] \subseteq \Sigma_3^* \Z [\langle \beta \rangle \times \langle \tau \rangle ]$ and the opposite inlcusion holds as well.
  Notice also that by equation \ref{ramification2} only the powers $\tau^i, i=0,\ldots,\frac n2 -1$ will survive in the ramified case. The result follows.
\end{proof} 
We have that $\operatorname{Im}(Q)$ equals to the space generated by elements
\[
\begin{pmatrix} 
  \Sigma_1 r_1\\
  \Sigma_2 r_2\\
  \Sigma_3^* r_3
\end{pmatrix} + \begin{pmatrix} 
  1 \\
  \alpha \\ 
  \alpha \beta 
\end{pmatrix} r_4. \] For $r_1,\ldots,r_4 \in \Z[H_n]$ the first summand forms a free $\Z$-module of rank $3n^2$ (resp. $\frac52 n^2$) for the unramified (resp. ramified) case and the second summand is a free $\Z$-module of rank $n^3$. Furthermore, their intersection is $\Z$.

Indeed, suppose we have $r_1,\ldots,r_4 \in \Z[H_n]$ such that 
\[ (\Sigma_1 r_1, \Sigma_2 r_2, \Sigma_3^* r_3) = r_4 (1,\alpha,\alpha\beta)\] then by comparing the coordinates, $r_4 = \Sigma_1 r_1$ which implies $\alpha$ acts trivially on $r_4$. By the second coordinates, $\Sigma_2 r_2 = \alpha r_4 = r_4$ which implies $\beta$ acts trivially on $r_4$. Similarly, by the third coordinates and what we have already established, $\alpha\beta$ acts trivially on $r_4$. This implies that $r_4$ is invariant under the action of the whole group $H_n$, that is $r_4$ belongs to the rank one $\Z$-module generated by $\Sigma_1\Sigma_2\Sigma^*_3$. We have proved that 

\begin{lemma} \[\operatorname{Im}(Q) = \left( \bigoplus\limits_{\nu=1}^2 \Sigma_i \Z[H_n]\right) \bigoplus \Sigma_3^* \Z[H_n] \bigoplus \Z[H_n]/\Z \Sigma_1\Sigma_2\Sigma_3^*.\]
  Also,
  \[\operatorname{rank}_\Z Q = \begin{cases}
    n^3 + 3n^2 - 1, & \textrm{ in the unramified case}, \\
    n^3 + \frac52 n^2 - 1, & \textrm{ in the ramified case}.
  \end{cases} \]  
\end{lemma}

\begin{remark}
  \label{rem:ind}
  If $\Sigma$ is an element invariant by the action of a subgroup $H < G$, then $\Sigma H= \mathrm{Ind}_H^G \mathbb{F}$, where $\mathbb{F}$ has the trivial $H$ action. Indeed, observe that $\mathrm{Ind}_H^G \mathbb{F} = \mathbb{F} \otimes_{ \mathbb{F}[H]} \mathbb{F}[G]= \Sigma \mathbb{F}[G]$.  
\end{remark}

The module $\Sigma_1 \mathbb{F}[H_n]$ is isomorphic to $\Sigma_1 \mathbb{F}[\langle \beta \rangle \times \langle \tau \rangle]$ by the proof of lemma \ref{sigmas}. 
Using remark \ref{rem:ind} we see that 
 as a representation its 
character $\chi_{\Sigma_1}$ is the character of $\mathrm{Ind}_{\langle \alpha \rangle}^{H_n}   \mathbb{F}$.
Denote by $p_\alpha$ the trivial character of the group $\langle \alpha \rangle$. 
 Using Frobenius reciprocity we compute:
\begin{align*}
  \langle \chi_{ijs}, \chi_{\Sigma_1} \rangle_{H_n} &= \langle \operatorname{Res}\chi_{ijs}, p_\alpha \rangle_{ \langle \alpha \rangle } = \frac{1}{n} \sum\limits_{k=0}^{n-1}\chi_{ijs}(\alpha^k) \\
  &= \frac{1}{\gcd(n,j)}\sum\limits_{k=0}^{\gcd(n,j)-1} \zeta^{k \frac n{\gcd(n,j)}i} =\begin{cases} 1, & i=0 ,\\ 0, & i\neq 0. \end{cases}
\end{align*} We have computed that
\[\Sigma_1 \mathbb{F}[H_n] = \bigoplus\limits_{j=0}^{n-1} \bigoplus\limits_{s=0}^{\gcd(n,j)-1} \mathbb{F} \chi_{0js}.\] Notice that the above consists of the $n$ one-dimensional subspaces ($j=0$) and of $(n-1)\cdot \gcd(n,j)$ subspaces ($j\neq 0$) of dimension $n/\gcd(n,j)$, adding up to the correct dimension $n^2$. 

A similar computation yields that

\[\Sigma_2 \mathbb{F}[H_n]= \bigoplus\limits_{j=0}^{n-1}\bigoplus\limits_{i=0}^{\gcd(n,j)-1} \mathbb{F} \chi_{ij0}.\] 

In the module $\Sigma_3^* \mathbb{F}[H_n]$, we have that $\alpha \beta$ is annihilated, therefore 
$\Sigma_3^* \mathbb{F}[H_n] = \mathrm{Ind}_{\langle \alpha \beta \rangle}^{H_n} \mathbb{F}$. 
Moreover,  we can expect the one dimensional irreducible characters $\chi_{i0s}$ that appear in $\chi_{\Sigma^*_3}$ map 
$\alpha\beta = \beta \tau \alpha$ to $1$, that is $i+s\equiv 0 \mod n$. We will see how this generalizes to the higher dimensional irreducible characters using Frobenius reciprocity. 
Let again $p_{\alpha\beta}$ 
be the trivial  character of $\mathbb{F}[\langle \alpha \beta \rangle]$. We set $d_j = \gcd(n,j)$ and use the fact that $(\alpha\beta)^k = \beta^k \alpha^k \tau^{\binom{k+1}{2}}$. We compute
\begin{align*}
\langle \chi_{ijs}, \chi_{\Sigma_3^*} \rangle_{H_n} &= \langle \operatorname{Res} \chi_{ijs}, p_{\alpha \beta} \rangle_{\langle \alpha\beta \rangle} = \frac1{|\langle \alpha \beta \rangle |} \sum\limits_{k=0}^{|\langle \alpha \beta \rangle |-1} 
\chi_{ijs}( (\alpha\beta)^k ) \\
&= \begin{cases} 
  \frac{1}{d_j}\sum\limits_{k=0}^{d_j-1} \zeta^{k\frac{n}{d_j}(i+s) + j\binom{\frac{kn}{d_j} + 1 }{2}}, &  2\nmid n,\\
  \frac{1}{2d_j}\sum\limits_{k=0}^{2d_j-1} \zeta^{k\frac{n}{d_j}(i+s) + j\binom{\frac{kn}{d_j} + 1 }{2}}, &  2\mid n.
\end{cases}
\end{align*} For odd $n$ the quantity $j \binom{\frac{kn}{d_j}+1}{2}$ is divisible by $n$, thus
\[\langle \chi_{ijs}, \chi_{\Sigma_3} \rangle_{H_n} = \frac{1}{d_j}\sum\limits_{k=0}^{d_j-1} \zeta^{k\frac{n}{d_j}(i+s)} = \begin{cases}
  1, & i+s \equiv 0 \mod d_j, \\
  0, & i+s \not\equiv 0 \mod d_j.
\end{cases}\]
We deal with $n$ being even now. Using that $(ab)^k = (ab)^{k^\prime} \tau^{\frac{n}2}$ for $k\geq \frac n2, k\equiv k^\prime \mod \frac n2$, we have that 
\[\langle \chi_{ijs}, \chi_{\Sigma_3^\prime} \rangle_{H_n} = \frac{1}{2d_j}(1+\zeta^{j\frac n2})\sum\limits_{k=0}^{d_j-1} \zeta^{k\frac{n}{d_j}(i+s) + j\binom{\frac{kn}{d_j} + 1 }{2}},\] which is $0$ if $j$ is not even. We thus assume $j$ is even now, it easy to see that
\[j \binom{\frac{kn}{d_j}+1}{2} \equiv \begin{cases}  0 \mod n, & 2\nmid \frac{n}{d_j}, \\
  k \frac{j}2 \frac{n}{d_j} \mod n, & 2\mid \frac{n}{d_j},
\end{cases} \]  from which we deduce
\[\langle \chi_{ijs}, \chi_{\Sigma_3^\prime} \rangle_{H_n} = \begin{cases} 
  1, & i+s \equiv 0 \mod d_j, \; 2\nmid \frac{n}{d_j},\\
  1, & i+s \equiv -\frac{j}{2} \mod d_j, \; 2 \mid \frac{n}{d_j}, \\
  0, & \textrm{otherwise}.
\end{cases}\] Let $j^\prime$ be $0$ or $-\frac{j}{2}$ according to the previous statement. We have proved that
\[\Sigma_3^* \mathbb{F}[H_n] = \begin{cases} 
  \bigoplus\limits_{j=0}^{n-1} \bigoplus\limits_{\substack{i,s=0 \\ i+s \equiv 0 \mod \gcd(n,j)}}^{\gcd(n,j)-1} \mathbb{F} \chi_{ijs}, & 2\nmid n,  \\
  & \\
  \bigoplus\limits_{\substack{j=0 \\ 2\mid j}}^{n-1} \bigoplus\limits_{\substack{i,s=0 \\ i+s \equiv j^\prime \mod \gcd(n,j)}}^{\gcd(n,j)-1} \mathbb{F} \chi_{ijs}, & 2\mid n,
\end{cases}\] adding up to the correct dimensions $n^2$ and $n^2/2$ respectively.

Additionally, the module $\mathbb{F}[H_n]/ \Sigma_1 \Sigma_2 \Sigma^*_3$ has 
every possible character $\chi_{ijs}$ except for those induced by $\chi_{ij}\otimes \chi_s$ which map all elements $\alpha,\beta,\alpha\beta$ to $1$, which only happens for $(i,j,s)=(0,0,0)$. More specifically, $\mathbb{F}[H_n]/ \Sigma_1 \Sigma_2 \Sigma^*_3$ has every possible character except for  $\chi_{000}$.

Counting all the characters that appear previously, we set $z_j(i,s)$ as the number of times the character $\chi_{ijs}$ appears on all $\Sigma_1 \mathbb{F}[H_n], \Sigma_2 \mathbb{F}[H_n]$ 
and $\Sigma_3^* \mathbb{F}[H_n]$, that is in compact notation, 
\[z_j(i,s)= \begin{cases} [i=0 \textrm{ or } s=0] + [i=s=0] + [i+s\equiv 0 \mod \gcd(n,j)], & 2\nmid n, \\
  [i=0 \textrm{ or } s=0] + [i=s=0] + [i+s\equiv j^\prime \mod \gcd(n,j)], & 2\mid n,j  \\
  [i=0 \textrm{ or } s=0] + [i=s=0], & 2 \mid n, 2 \nmid j,
\end{cases}\] where $[P]$ is $1$ if property $P$ holds and $0$ otherwise, with the convention that all equivalences are satisfied $\mod 1$.

\begin{lemma} \label{imQdecomp}
  We have the following decomposition
  \[\operatorname{Im}(Q) \otimes \mathbb{F} = \bigoplus \mathbb{F} c_{ijs} \chi_{ijs} \] 
  where 
  \[c_{ijs} = \begin{cases}
    \frac{n}{\gcd(n,j)} + z_j(i,s), & \textrm{ if } (i,j,s)\neq (0,0,0), \\
    \frac{n}{\gcd(n,j)}-1 + z_j(i,s) = 3, & \textrm{ if } (i,j,s) = (0,0,0).
  \end{cases}\]
\end{lemma}

\begin{lemma} \label{Adecomp} We have the decomposition of the Alexander module
  \[\mathcal{A}_\psi \otimes \mathbb{F} = \bigoplus\limits_{j=0}^{n-1}\bigoplus\limits_{i,s=0}^{\gcd(n,j)-1} \mathbb{F} a_{ijs} \chi_{ijs},\] where 
  \[a_{ijs} =  \left(3\frac{n}{\gcd(n,j)} - c_{ijs} \right).\] Also, 
  
  \[\operatorname{rank}_\Z \mathcal{A}_\psi = \begin{cases}
    2n^3 - 3n^2 + 1, &\textrm{ in the unramified case}, \\
    2n^3 -\frac52 n^2 +1, &\textrm{ in the ramified case},
  \end{cases}\] agreeing with $2g +n^3-1$ that we have already counted from the Crowell exact sequence.
\end{lemma}
\begin{proof} The Alexander module $\mathcal{A}_\psi$ is the cokernel of $Q$, thus
  \begin{align*}
    \operatorname{rank}_\Z \mathcal{A}_\psi &= \operatorname{rank}_\Z \Z[H_n]^3 - \operatorname{rank}_\Z Q\ \\
    &= 3n^3 - \begin{cases}
      n^3 + 3n^2 - 1, &  \\
      n^3 + \frac52 n^2 - 1, & .
    \end{cases} \\
    &= \begin{cases}
      2n^3 - 3n^2 + 1, &\textrm{ in the unramified case}, \\
      2n^3 -\frac52 n^2 +1, &\textrm{ in the ramified case},
    \end{cases}.
  \end{align*} and the decomposition of $\mathcal{A}_\psi \otimes \mathbb{F}$ follows from the decomposition of $\mathbb{F}[H_n]$ into characters.
\end{proof}

\begin{theorem}
  \label{thm:homHeiss}
  \[H_1(X_H,\mathbb{F}) = \bigoplus\limits_{j=0}^{n-1} \;\;\bigoplus\limits_{i,s=0}^{\gcd(n,j)-1} \mathbb{F} h_{ijs} \chi_{ijs},\]
  where
\begin{equation}
\label{eq:Hom-coeffs}
h_{ijs} = \begin{cases} \frac{n}{\gcd(n,j)} - z_j(i,s), & \textrm{ if } (i,j,s) \neq (0,0,0), \\
  0, &  \textrm{ if } (i,j,s) = (0,0,0).
\end{cases}
\end{equation}
\end{theorem}
\begin{proof} Follows from the lemma \ref{Adecomp} and the Crowell exact sequence as in (\ref{Crowell2}). More specifically, for every non-principal character we have that
  \[h_{ijs} = a_{ijs}-\frac{n}{\gcd(n,j)} = 2\frac{n}{\gcd(n,j)} - c_{ijs} = \frac{n}{\gcd(n,j)} - z_j(i,s).\] 
\end{proof}

\appendix 

\section{ \\ Irreducible Representations of $H_n$} \label{appendix}
In this section we will compute the irreducible representations of the Heisenberg group $H_n$. This is by no means a new result and it can be found for instance in the work of J. Grassberger and 
G. Hörmann in \cite{MR1835586}. We rely on the general method regarding semi-direct products of groups with the normal group being abelian, as it is described by Serre in section 8.2 of his book \cite{SerreLinear}.

Indeed, we can realize $H_n$ as the semi direct product $\langle \alpha, \tau \rangle \rtimes \langle \beta \rangle$ since every element of $H_n$ can be expressed as $\beta^j \tau^k \alpha^i$, 
$0\leq i,j,k \leq n-1$ with $\tau = [\alpha,\beta]$ commuting with both $\alpha$ and $\beta$. As $H_n$ is isomorphic to $(\Z / n \Z)^2 \rtimes \Z/n\Z$ we 
denote by $\chi_{ij}$ and $\chi_s$ the irreducible characters of the first and the second component respectively, such that \[ \chi_{ij}(\alpha) = 
\zeta^i, \ \chi_{ij}(\tau) = \zeta^j, \ \chi_s (\beta) = \zeta^s, \quad 0\leq i,j,s \leq n-1, \] for a fixed primitive $n$-th root of unity $\zeta$. 

Applying the method in Serre's book, we can conclude that the irreducible representations of $H_n$ are those induced by the product $\chi_{ij}\otimes \chi_s$ for the indices $0\leq i,s \leq \gcd(n,j)-1$ and $ 0\leq j \leq n-1$. We will denote them by 
$\chi_{ijs}$ and for a $g$ in $H_n$ its image is the 
linear combination of the values of $\chi_{ij}\otimes \chi_s$ over the elements in the intersection of the conjugacy class of $g$ and the group $G_{ij} = \langle \alpha ,\tau \rangle \cdot \langle \beta^{ \frac n {\gcd(n,j)}} \rangle $. We have explicitly

\[ \chi_{ijs} (\beta^m \tau^\lambda \alpha^\mu ) = \sum\limits_{\substack{\nu=0 \\ \tau^{\lambda + \nu\mu}\beta^m \alpha^\mu \in G_{ij}}}^{n/\gcd(n,j)-1 } \chi_{ij}(\tau^{\lambda + \nu\mu} \alpha^\mu) \cdot \chi_s(\beta^m),\] 
for $0 \leq j \leq n-1$ and $0\leq i,s \leq \gcd(n,j)-1$, with each $\chi_{ijs}$ being of dimension $\frac n {\gcd(n,j)}$. Set $d_j = \gcd(n,j)$, the characters $\chi_{ijs}$ take values as following:
\[\chi_{ijs}(\beta^m \tau^\lambda \alpha^\mu ) = \begin{cases} 
  \frac{n}{d_j } \zeta^{\mu i + m s + \lambda j},  & \frac{n}{d_j}\mid m, \mu, \\
  0, & \textrm{otherwise},
\end{cases} \]

which we use for computations with the Alexander modules. 

As the number of irreducible representations is equal to the number of conjugacy classes, we verify the above by counting the conjugacy classes. The conjugacy class of $\beta^j \tau^k \alpha^i$ 
consists of the elements $\beta^j \tau^{k + iv - ju} \alpha^i$, or in terms of matrices 

\[ \begin{pmatrix}
  1 & -u & uv -z \\
  0 & 1 & -v \\
  0 & 0 & 1 
\end{pmatrix} \cdot \begin{pmatrix} 
  1 & i & k \\
  0 & 1 & j \\
  0 & 0 & 1
\end{pmatrix} \cdot \begin{pmatrix} 
  1 & u & z \\
  0 & 1 & v \\
  0 & 0 & 1
\end{pmatrix}  = \begin{pmatrix} 
  1 & i & k + iv - ju \\
  0 & 1 & j \\
  0 & 0 & 1
\end{pmatrix}\]

As argued in \cite{MR1835586} two elements $\beta^j \tau^k \alpha^i, \beta^{j^\prime} \tau^{k^\prime} \alpha^{i^\prime}$ are in the same class iff $i=i^\prime, j = j^\prime$ and 
$k = k^\prime + iv -ju \mod n$ which has a solution iff $k = k^\prime \mod \gcd(i,j,n)$. Thus, each conjugacy class contains exactly one element $\beta^j \tau^k \alpha^i$ with $k < 
\gcd(n,j,i)$. These are in total $\sum\limits_{i,j=0}^{n-1}\gcd(n,j,i)$, which is equal to the sum $\sum\limits_{j=0}^{n-1} \gcd(n,j)^2$ we counted according to Serre, as both are equal to $\sum\limits_{d \mid n} d^2 \phi(\frac nd)$, where $\phi$ is the Euler's totient function.



\bibliographystyle{plain}

\begin{thebibliography}{10}

\bibitem{MR4252293}
Jannis~A. Antoniadis and Aristides Kontogeorgis.
\newblock The group of automorphisms of the {H}eisenberg curve.
\newblock In {\em Abelian varieties and number theory}, volume 767 of {\em Contemp. Math.}, pages 25--39. Amer. Math. Soc., [Providence], RI, [2021] \copyright 2021.

\bibitem{BirmanBraids}
Joan~S. Birman.
\newblock {\em Braids, links, and mapping class groups}.
\newblock Princeton University Press, Princeton, N.J.; University of Tokyo Press, Tokyo, 1974.
\newblock Annals of Mathematics Studies, No. 82.

\bibitem{MR1816711}
Daniel~K. Biss and Samit Dasgupta.
\newblock A presentation for the unipotent group over rings with identity.
\newblock {\em J. Algebra}, 237(2):691--707, 2001.

\bibitem{bogoGrp}
Oleg Bogopolski.
\newblock {\em Introduction to group theory}.
\newblock EMS Textbooks in Mathematics. European Mathematical Society (EMS), Z\"urich, 2008.
\newblock Translated, revised and expanded from the 2002 Russian original.

\bibitem{DebesDouai97}
Pierre D{\`e}bes and Jean-Claude Douai.
\newblock Algebraic covers: field of moduli versus field of definition.
\newblock {\em Ann. Sci. \'Ecole Norm. Sup. (4)}, 30(3):303--338, 1997.

\bibitem{DebesEmsalem}
Pierre D{\`e}bes and Michel Emsalem.
\newblock On fields of moduli of curves.
\newblock {\em J. Algebra}, 211(1):42--56, 1999.

\bibitem{DDMS}
J.~D. Dixon, M.~P.~F. du~Sautoy, A.~Mann, and D.~Segal.
\newblock {\em Analytic pro-{$p$} groups}, volume~61 of {\em Cambridge Studies in Advanced Mathematics}.
\newblock Cambridge University Press, Cambridge, second edition, 1999.

\bibitem{MR1835586}
Johannes Grassberger and G\"unther H\"ormann.
\newblock A note on representations of the finite {H}eisenberg group and sums of greatest common divisors.
\newblock {\em Discrete Math. Theor. Comput. Sci.}, 4(2):91--100, 2001.

\bibitem{Ihara1985-it}
Yasutaka Ihara.
\newblock Profinite braid groups, {G}alois representations and complex multiplications.
\newblock {\em Ann. of Math. (2)}, 123(1):43--106, 1986.

\bibitem{IharaCruz}
Yasutaka Ihara.
\newblock Arithmetic analogues of braid groups and {G}alois representations.
\newblock In {\em Braids ({S}anta {C}ruz, {CA}, 1986)}, volume~78 of {\em Contemp. Math.}, pages 245--257. Amer. Math. Soc., Providence, RI, 1988.

\bibitem{MR4186523}
Aristides Kontogeorgis and Panagiotis Paramantzoglou.
\newblock Galois action on homology of generalized {F}ermat curves.
\newblock {\em Q. J. Math.}, 71(4):1377--1417, 2020.

\bibitem{MR4117575}
Aristides Kontogeorgis and Panagiotis Paramantzoglou.
\newblock Group {A}ctions on cyclic covers of the projective line.
\newblock {\em Geom. Dedicata}, 207:311--334, 2020.

\bibitem{Morishita2011-yw}
M~Morishita.
\newblock {\em Knots and Primes: An Introduction to Arithmetic Topology}.
\newblock SpringerLink : B{\"{u}}cher. Springer-Verlag London Limited, 2011.

\bibitem{SerreLinear}
Jean-Pierre Serre.
\newblock {\em Linear representations of finite groups}.
\newblock Springer-Verlag, New York, 1977.
\newblock Translated from the second French edition by Leonard L. Scott, Graduate Texts in Mathematics, Vol. 42.

\bibitem{sagemath9.8}
{The Sage Developers}.
\newblock {\em {S}ageMath, the {S}age {M}athematics {S}oftware {S}ystem ({V}ersion 9.8)}, 2023.
\newblock {\tt https://www.sagemath.org}.

\end{thebibliography}



\def\cprime{$'$}

\end{document}